\documentclass[12pt, reqno]{amsart}
\usepackage{amsmath, amsthm, amscd, amsfonts, amssymb, graphicx, color}
\usepackage[bookmarksnumbered, colorlinks, plainpages]{hyperref}
\hypersetup{colorlinks=true,linkcolor=red, anchorcolor=green, citecolor=cyan, urlcolor=red, filecolor=magenta, pdftoolbar=true}

\usepackage{tensor}
\usepackage{mathrsfs}
\usepackage{cleveref}

\newtheorem{theorem}{Theorem}[section]

\newtheorem{proposition}[theorem]{Proposition}
\newtheorem{corollary}[theorem]{Corollary}
\theoremstyle{definition}
\newtheorem{definition}[theorem]{Definition}
\newtheorem{example}[theorem]{Example}

\theoremstyle{remark}
\newtheorem{remark}[theorem]{Remark}
\numberwithin{equation}{section}

\usepackage{fullpage}

\begin{document}

\setcounter{page}{1}

\title[FP on PMTs]{Fixed points on partial metric type spaces }

\author[Ya\'e Ulrich Gaba]{Ya\'e Ulrich Gaba$^{1,2,3,*}$}



\address{$^{1}$ Institut de Math\'ematiques et de Sciences Physiques (IMSP), 01 BP 613 Porto-Novo, B\'enin.}


\address{$^{2}$ Department of Mathematical Sciences, North West University, Private Bag
	X2046, Mmabatho 2735, South Africa.}

\address{$^{3}$ African Center for Advanced Studies (ACAS),
	P.O. Box 4477, Yaounde, Cameroon.}
\email{\textcolor[rgb]{0.00,0.00,0.84}{gabayae2@gmail.com
}}




\begin{abstract}

	In this paper, we study some new fixed point results for self maps defined on partial
	metric type spaces.  In particular, we give common fixed point theorems in the same setting. Some examples are given
	which illustrate the results.

\end{abstract} 

\maketitle

\section{Introduction and preliminaries}

In the 1920s, Banach proposed and proved his famous Contraction Principle. Over the years, many researchers have proposed new concepts of contraction mapping as well as the new fixed point theorems they induce. In 1994, Matthews\cite{mat} introduced the concept of partial metric space and
proved the Banach Contraction Principle in these spaces. Partial metrics generalize the concept of a metric space and are useful in modelling partially defined
information, which often appears in computer science. The particularity of these spaces is the property that the self-distance of any point of the space may not be zero. More fixed point theory of partial metric
space have been proposed, for example, see \cite{h,sal}. In 2014, S. Satish introduced the concept of partial
$b$-metric space \cite{shu}, and the fixed point theorem of Banach Contraction Principle and Kannan type mapping were proved in partial $b$-metric space. Another generalization of metric spaces, namely metric type spaces, has been discussed by Gaba et al. \cite{gab,niyi-gaba} and here we shall establish a connection between these spaces and the so-called partial metric type spaces, that we introduce in the coming lines. In this paper, we also prove some common fixed point theorem for contractive mappings in partial metric type space which generalize and extend the result of S. Satish et al. \cite{shu3,shu,shu2} (in particular). In this line, an analog of the Banach contraction principle as well as the Kannan type fixed point theorem in partial $b$-metric spaces are given. Some examples are given
which illustrate the results.

First, we introduce the definition of $K$-partial metric spaces.

\begin{definition} (Compare \cite[Definition 3.]{shu})
	A partial metric type on a set $X$ is a function $p: X \times X \to [0, \infty)$ such that:
	\begin{enumerate}
		\item[(pm1)] $x = y$ iff $(p(x, x) = p(x, y)=p(y,y)$ whenever $x, y \in X$,
		\item[(pm2)] $p(x, x)\leq p(x, y)$ whenever $x, y \in X$,
\item[(pm3)] $p(x, y) = p(y, x);$ whenever $x, y \in X$,

		\item[(pm4)] there exists a real number $K \geq 1$ such that $$p(x, y) + p(z_1,z_1)+p(z_2,z_2)+\cdots+p(z_n,z_n) \leq K[ p(x,z_1)+p(z_1,z_2)+\cdots+p(z_n,y)]$$ 
		for any points $x,y,z_i\in X,\ i=1,2,\ldots, n$ where $n\geq 1$ is a fixed natural number .
	\end{enumerate}

	The triplet $(X, p,K)$ is called a $K$-partial metric space ( or $K$-PMS for short). The number $K$ is
	called the coefficient of $(X, p, K)$.
	
\end{definition}
In the axiom (pm4), if we take $n=1$, we recover the definition of a partial $b$-metric space (in the sense of Satish). Also, for $n=2,K=1$, we recover the definition of a partial rectangular metric spaces in the sense of \cite[Definition 3.]{shu2}. Hence partial metric type spaces comprise both partial $b$-metric spaces and partial rectangular metric spaces.

\begin{remark}
It is clear that, if $p ( x , y ) = 0$ , then, from (pm1) and (pm2), $x = y$. But if $x = y$, $p ( x , y )$ may not be 0. A basic example of a $K$-PMS is the triplet $(\mathbb{R}^+ , p,K )$ where $p(x, y) = [\max\{x, y\}]^s + |x - y|^s$ for all $x, y \in \mathbb{R}^+$, $s>1$. It is not difficult to see that therefore, $p$ is a $K$-PMS with $K=2^s >1,$ but not a partial metric on $\mathbb{R}^+$.
Other examples of $K$-PMS can be constructed, in particular, Satish \cite[Proposition 1., Proposition 2.]{shu} describes such a construction. More precisely, he proved that it is enough to add a partial metric and a $b$-metric in order to obtain a partial $b$-metric. Also, a power $q\geq 1$ of a partial metric is a $K$-PMS.
\end{remark}

Now, we define Cauchy sequence and convergent sequence in partial metric type spaces (see \cite{shu}).

Every partial $K$-PMS $p$ on a nonempty set $X$ generates a topology $\tau_p$ on $X$ whose base is the family of open $p$-balls $B_p(x,\varepsilon)$ where $$\tau_b = \left\lbrace B_p(x,\varepsilon), x\in X, \varepsilon>0  \right\rbrace$$
and $$B_p(x,\varepsilon) = \{y \in X : b(x, y) < \varepsilon + b(x, x)\}.$$

The topological space $(X, \tau_b )$ is $T_0$, but not $T_1$ in general, talk less of being $T_2$, hence limit of convergent sequence
may not be unique.
\begin{definition}
	Let $(X, p,K)$ be a $K$-PMS and $(x_n)_{n\geq 1}$ be any sequence in $X$ and $x \in X$. Then:
	
	\begin{enumerate}
		\item The sequence $(x_n)_{n\geq 1}$ is said to be convergent with respect to $\tau(p)$ (or $\tau(p)$-convergent) and converges to $x$, if $\lim\limits_{n\to \infty} p(x_n , x) = p(x, x)$. Alternatively, we could also write $x_n \overset{p}{\longrightarrow} x .$

		
		\item The sequence $(x_n)_{n\geq 1}$ is said to be $\tau(p)$-Cauchy (or just Cauchy) if
		$$\lim\limits_{n,m\to \infty} p(x_n , x_m)$$ exists and is finite.
		
		
		
		\item $(X, p,K)$ is said to be $\tau(p)$-complete (or just complete) if for every $\tau(p)$-Cauchy
		sequence $(x_n)_{n\geq 1} \subseteq X$, there exists $x \in X$ such that:
		
		$$ \lim\limits_{n,m\to \infty} p(x_n , x_m)= \lim\limits_{n\to \infty} p(x_n , x)=p(x,x).$$


		

	\end{enumerate}
	
\end{definition}

We give these additional definitions, useful to characterize some specific complete partial metric type spaces.
\begin{definition}
	Let $(X, p,K)$ be a $K$-PMS and $(x_n)_{n\geq 1}$ be any sequence in $X$ and $x \in X$. Then:
	
	\begin{enumerate}
		\item The sequence $(x_n)_{n\geq 1}$ is called $0$-Cauchy if
		$$\lim\limits_{n,m\to \infty} p(x_n , x_m)=0.$$ 
		
		\item $(X, p,K)$ is called $0$-complete  if for every $0$-Cauchy
		sequence $(x_n)_{n\geq 1} \subseteq X$, there exists $x \in X$ such that:
		
		$$ \lim\limits_{n,m\to \infty} p(x_n , x_m)= \lim\limits_{n\to \infty} p(x_n , x)=p(x,x)=0.$$
		
		\end{enumerate}
\end{definition}

\begin{remark}
It is straightforward from the definition that  if the partial metric type $(X,p,K)$ is complete, then it is $0$-complete. The following example shows that the converse need not hold.
\end{remark}

\begin{example}
	Let $X=(0,1)$ and $p(x,y)=|x-y|^2+2$ for all $x,y \in X$. The $(X,p,2)$ is a $0$-complete partial metric space which fails to be complete. Indeed, since $p(x,y) \geq 2$ for all $x,y \in X$, $(X,p,K)$
	has no $0$-Cauchy, this proves that $(X,p)$ is $0$-complete. By 
	$$\lim\limits_{n,m\to \infty}p\left( \frac{1}{2m},\frac{1}{2n}\right)= \lim\limits_{n,m\to \infty} \left| \frac{1}{2m}-\frac{1}{2n} \right|^2+2 = 2,$$
	we conclude that the sequence $\left(\frac{1}{2n}\right)_{n\geq 1}$ is a Cauchy sequence in $(X,p,2)$. If we assume that there is $x^* \in X$ such that $x_n \overset{p}{\longrightarrow} x^*$, then 
	
	$$\lim\limits_{m\to \infty}p\left( \frac{1}{2m},x^*\right)= \lim\limits_{m\to \infty} \left| \frac{1}{2m}-x^* \right|^2+2 =p(x^*,x^*)=2.$$
	It implies that $x^*= 0 \in X$ --a contradiction. Hence $\left(\frac{1}{2n}\right)_{n\geq 1}$ is not convergent in $(X,p,2)$ which is therefore not complete. 
\end{example}

Next, we recall the definition of a metric type space.

\begin{definition}(See \cite[Definition 1.1]{gab})
	Let $X$ be a nonempty set, and let the function $D:X\times X \to [0,\infty)$ satisfy the following properties:
	\begin{itemize}
		\item[(D1)] $D(x,x)=0$ for any $x \in X$;
		\item[(D2)] $D(x,y)=D(y,x)$ for any $x,y\in X$;
		\item[(D3)] $D(x,y) \leq K \big( D(x,z_1)+D(z_1,z_2)+\cdots+D(z_n,y) \big)$ for any points $x,y,z_i\in X,\ i=1,2,\ldots, n$ where $n\geq 1$ is a fixed natural number and some constant $K\geq 1$.
	\end{itemize}
	The triplet $(X,D,K)$ is called a \textbf{metric type space}.
\end{definition}
The class of metric type spaces is strictly larger than that of metric spaces (see \cite{niyi-gaba}). The concepts of Cauchy sequence, convergence for a sequence and completeness in a metric type space are defined in the same way as defined for a metric space.

In the axiom (D3), if we take $n=1$, we recover the definition of a $b$-metric space and for $n=2$, we recover the definition of a rectangular $b$-metric space (see \cite[Definition 1.3]{shu3}). Hence metric type spaces comprise both $b$-metric spaces and rectangular metric spaces.


\section{Some First Results}

We begin by showing that that every partial metric type induces a metric type.

\begin{proposition}\label{prop1}
	For each partial metric type space $(X,p,1)$, the triplet $(X,p^t,1)$\footnote{This is actually a classical metric space.} is a metric type space, where 
	
	$$ p^t(x,y) = 2 p(x,y) -p(x,x)-p(y,y) \quad \forall x,y\in X.$$
\end{proposition}

\begin{proof}
	By definition of $p$ and $p^t$, it is easy to verify that $p^t$ satisfies the properties (D1) and (D2). For (D3), let $z_i\in X,\ i=1,2,\ldots, n$ be $n$ points.
	\begin{align*}
		p^t(x,y) &= 2 p(x,y) -p(x,x)-p(y,y)  \\
		         &\leq -2(p(z_1,z_1)+p(z_2,z_2)+\cdots+p(z_n,z_n)) \\
		         & +2 [ p(x,z_1)+p(z_1,z_2)+\cdots+p(z_n,y)]-p(x,x)-p(y,y)\\
		         & \leq 2 p(x,z_1) -p(x,x)-p(z_1,z_1) + 2p(z_1,z_2)-p(z_1,z_1,)-p(z_2,z_2)+\cdots \\
		         & + \cdots + 2p(z_n,y)-p(z_n,z_n)-p(y,y)\\
		         & \leq p^t(x,z_1)+ p^t(z_1,z_2)+\cdots p^t(z_n,y).
	\end{align*}
	Thus $(X,p^t,1)$ is a metric type space.
\end{proof}

The next proposition provides us with the converse of Proposition \ref{prop1}.

\begin{proposition}
Let $(X,d,1)$ be a metric type space. Whenever $x,y \in X, x\neq$ assume that there exists $x_0 \in X$ such that  $d(x_0,x) \leq d(x,y)$. Then, the triplet $(X,p,1)$ is a partial metric space type, where 
$$p(x,y) = \frac{1}{2}[d(x,y)+ d(x,x_0)+d(y,x_0)]\quad \forall x,y\in X.$$
\end{proposition}

\begin{proof}
	Axioms (pm2) and (pm3) follow by the definition of $p$. For (pm1), $p(x,x) = d(x,x_0)$ for all $x\in X,$ hence for all $x,y\in X$, $ \frac{1}{2}[d(x,y)+ d(x,x_0)+d(y,x_0)] = d(x,x_0) = d(y,x_0)$, i.e. $d(x,y)=0$ and it follows that $x=y.$ For property (pm4), let $z_i\in X,\ i=1,2,\ldots, n$ be $n$ points. By definition of $p$, we have
	
	\begin{align*}
		p(x,y) & = \frac{1}{2}[d(x,y)+ d(x,x_0)+d(y,x_0)]\\
		       & \leq \frac{1}{2}[ d(x,z_1)+d(z_1,z_2)+\cdots+d(z_n,y)+ d(x,x_0)+d(y,x_0) ] \\
		       & \leq \frac{1}{2}[d(x,z_1) + d(x,x_0)+ d(z_1,x_0) + \cdots +\\
		       & + d(z_n,y)+d(z_n,x_0)+d(y,x_0)]-d(z_1,x_0)-\cdots \\
		       & - d(z_n,x_0)\\
		       & = p(x,z_1)+p(z_1,z_2)+\cdots + p(z_n,y)- \sum_{i=1}^{n}p(z_i,z_i).
	\end{align*}
	Thus $(X,p,1)$ is a partial metric type space.
\end{proof}

%
%
%
%
%

The following result outlines a equivalence between the $0$-completeness of a $K$-PMS and the completeness of a metric type space.

\begin{proposition}
Let $(X,p,K)$
be a partial metric type space and $d_p(x,y) = 0$ if $x=y$ and $d_p(x,y)=p(x,y)$ if $x\neq y.$ Then we have

\begin{enumerate}
	\item $d_p$
	is a metric type on $X.$
	
	\item The partial metric type space $(X,p,K)$
	is $0$-complete if and only the metric type space
	$(X,d_p,K)$
	is complete.
\end{enumerate}
\end{proposition}

\begin{proof}
\begin{enumerate}
	\item It is easy to see that $d_p$
	is a metric type on $X.$
	
	\item Suppose that the partial metric type space $(X,p,K)$ is $0$-complete. For each Cauchy sequence in the metric space $(X,d_p,K)$, we
	may assume that $x_n \neq x_m$ for $n\neq m \in \mathbb{N}$. Then we have
	
	$$\lim\limits_{n,m \to \infty} p(x_n,x_m) = d_p(x_n,x_m)=0. $$
	This proves that $(x_n)_{n\geq 1}$
	is a $0$-Cauchy sequence in the $0$-complete partial
	metric type space $(X,p,K)$. Then there exists $x \in X$ such that
	$$ \lim\limits_{n,m\to \infty} p(x_n , x_m)= \lim\limits_{n\to \infty} p(x_n , x)=p(x,x)=0.$$
	
	It implies that
	$$ \lim\limits_{n\to \infty} p(x_n , x)=\lim\limits_{n\to \infty} d_p(x_n , x)=0.$$
This proves that $\lim\limits_{n\to \infty} x_n = x$
in the metric type space $(X,d_p,K)$, which is therefore complete.	
	
	Conversely, suppose that the metric type space $(X,d_p,K)$ complete. For each $0$-Cauchy sequence  $(x_n)_{n\geq 1}$ in the partial metric type $(X,p,K)$, we may assume that $x_n \neq x_m$ for $n\neq m \in \mathbb{N}$. We have $\lim\limits_{n,m\to \infty} p(x_n , x_m)=0$. This implies 
	$$ \lim\limits_{n,m\to \infty} d_p(x_n , x_m)=\lim\limits_{n,m\to \infty} p(x_n , x_m)=0.$$
	
	Thus $(x_n)_{n\geq 1}$ is a Cauchy sequence in the complete metric type space $(X,d_p,K)$. Then there exists
	$x\in X$ such that $\lim\limits_{n\to \infty} d_p(x_n , x)=0$. It implies $\lim\limits_{n\to \infty} p(x_n , x)=0$. Also, we have
	$$0\leq p(x,x)\leq \lim\limits_{n\to \infty} p(x,x_n)=0.$$
	
	Therefore 
	$$ \lim\limits_{n,m\to \infty} p(x_n , x_m)= \lim\limits_{n\to \infty} p(x_n , x)=p(x,x)=0.$$
	Then the partial metric type space 
	$(X,p,K)$
	is $0$-complete.
\end{enumerate}
\end{proof}


We begin by proving some fixed point theorems for a pair of contractive maps.

\begin{theorem}\label{theorem1}
	Let $(X,p,K)$  be a complete $K$-PMS and assume that the mappings $T_1,T_2: X \to X$ satisfy the contractive condition
	\[    
	p(T_1x,T_2y) \leq k p(x,y) \quad \text{ for all } x,y\in X,
	\]
	where $k\in [0,1)$ is a constant. Then $T_1$ and $T_2$ have a unique common fixed point $x^*\in X$ and for any $x\in X,$ the orbits $\{T_1^{2n+1}x, \ n\geq 1 \}$ and $\{T_2^{2n+2}x,\ n\geq 1 \}$ converge to the fixed point.
\end{theorem}

\begin{proof}
	
	For any arbitrary $x_0 \in X$ and $n\geq 1$ . Set $x_1=T_1x_0;\ x_2=T_2x_1;\ x_3=T_1x_2;\cdots;\ x_{2n+1}=T_1x_{2n};\ x_{2n+2}=T_2x_{2n+1}; \cdots .$
	
	We have 
	\begin{align*}
	p(x_{2n+1},x_{2n}) & = p(T_1x_{2n},T_2x_{2n-1}) \leq k p(x_{2n},x_{2n-1} ) \\
	& \leq k^2 p(x_{2n-1},x_{2n-2}) \leq \cdots \leq k^{2n} p(x_1,x_0).
	\end{align*}

So for $n>m,$ using (pm4), we can write
\begin{align*}
p(x_{2n},x_{2m}) & \leq  K [p(x_{2n},x_{2n-1})+ p(x_{2n-1},x_{2n-2})+\cdots p(x_{2m+1},x_{2m})] - \sum_{i=2m+1}^{2n-1}p(x_i,x_i) \\
& \leq K (k^{2n-1}+k^{2n-2}+\cdots + k^{2m})p(x_1,x_0)\\
& \leq K \frac{k^{2n}}{1-k}p(x_1,x_0).
\end{align*}

Similar upper bounds can also be obtained if one consider the terms $ p(x_{2n+1},x_{2m}), p(x_{2n+1},x_{2m+1})$.

This implies $p(x_n,x_m) \to 0$ as $n,m\to \infty$. Hence $(x_n)$ is a Cauchy sequence. From the completeness of $X$, we know that there is $x^* \in X$ such that

$$ 0=\lim\limits_{n,m\to \infty} p(x_n , x_m)= \lim\limits_{n\to \infty} p(x_n , x^*)=p(x^*,x^*).$$


For $n\geq 1$, we have

\begin{align*}
p(T_1x^*,x^*)& \leq K [p(T_1x^*, T_2x_{2n+1})+ p(T_2 x_{2n+1},x^*)]\\
& \leq K [k p(x^*, x_{2n+1}) +  p( x_{2n+2},x^*)] \to 0 \text{ as } n\to\infty .
\end{align*}

Hence $p(T_1x^*,x^*)=0,$ i.e. $T_1x^*=x^*$. In the same way it can be established that$T_2x^*=x^*$. Hence $T_1x^*=x^*=T_2x^*$. Thus $x^*$ is the common fixed point of pair of maps $T_1$ and $T_2$.
The uniqueness is easily read from the contractive condition. Indeed if $u, v \in X$ are two distinct common fixed points of $T_1$ and $T_2$, that is, $T_1u =T_2u= u, T_1v=T_2 v = v$.
$$p(u,v) = p(T_1u,T_2v) \leq k p(u,v) <  p(u,v),$$
--a contradiction. Therefore, we must have $p(u, v) = 0$, that is, $u = v.$
	
\end{proof}

\begin{corollary}
	Let $(X,p,K)$ be a complete partial metric type space and suppose that the mappings $T_1,T_2: X \to X$ satisfy for some positive integer $n$ the contractive condition
	\[    
	p(T_1^{2n+1}x,T_2^{2n+2}y) \leq k p(x,y) \quad \text{ for all } x,y\in X,
	\]
	where $k\in [0,1)$ is a constant. Then $T_1$ and $T_2$ have a unique common fixed point $x^*\in X$.
\end{corollary}

\begin{theorem}\label{theorem2}
	Let $(X,p,K)$ be a complete metric type space and the mappings $T_1,T_2: X \to X$ satisfy the contractive condition
	\[    
	p(T_1x,T_2y) \leq k [p(T_1x,x)+ p(y,T_2y)] \quad \text{ for all } x,y\in X,
	\]
	where $k\in [0,\alpha)$ with $\alpha < \min 1/K$ is a constant. Then $T_1$ and $T_2$ have a unique common fixed point $x^*\in X$ and for any $x\in X,$ the orbits $\{T_1^{2n+1}x, \ n\geq 1 \}$ $\{T_2^{2n+2}x,\ n\geq 1 \}$ converge to the fixed point.
\end{theorem}

\begin{proof}
For any arbitrary $x_0 \in X$ and $n\geq 1$ . Set $x_1=T_1x_0;\ x_2=T_2x_1;\ x_3=T_1x_2;\cdots;\ x_{2n+1}=T_1x_{2n};\ x_{2n+2}=T_2x_{2n+1}; \cdots .$

We have 
\begin{align*}
p(x_{2n+1},x_{2n}) & =p(T_1x_{2n},T_2x_{2n-1})\\
& \leq k [p(T_1x_{2n},x_{2n}) + p(x_{2n-1},T_2x_{2n-1})]\\
& = k [p(x_{2n+1},x_{2n}) + p(x_{2n},x_{2n-1})].\\
\end{align*}
So, 
\[
p(x_{2n+1},x_{2n}) \leq h p(x_{2n},x_{2n-1}) \text{ with } h=\frac{k}{1-k}.
\]

So for $n>m,$ we can write
\begin{align*}
p(x_{2n},x_{2m}) & \leq  K [p(x_{2n},x_{2n-1})+ p(x_{2n-1},x_{2n-2})+\cdots p(x_{2m+1},x_{2m})] \\
& \leq K (h^{2n-1}+h^{2n-2}+\cdots + h^{2m})p(x_0,x_1)\\
& \leq K \frac{h^{2n}}{1-h}p(x_0,x_1).
\end{align*}

This implies $p(x_n,x_m) \to 0$ as $n,m\to \infty$. Hence $(x_n)$ is a Cauchy sequence. From the completeness of $X$, we know that there is $x^* \in X$ such that

$$ 0=\lim\limits_{n,m\to \infty} p(x_n , x_m)= \lim\limits_{n\to \infty} p(x_n , x^*)=p(x^*,x^*).$$

For $n\geq 1$, we have

\begin{align*}
p(T_1x^*,x^*)& \leq K [p(T_1x^*, T_2x_{2n+1})+ p(T_2 x_{2n+1},x^*)]\\
& \leq K [k p(T_1x^*, x^*) +  kp(  x_{2n+1},T_2x_{2n+1}) + p(x_{2n+2},x^*) ]\\
& \leq K [k p(T_1x^*, x^*) +  kp(  x_{2n+1},x_{2n+2}) + p(x_{2n+2},x^*) ].
\end{align*}

Since $k<\frac{1}{K}$, therefore $p(T_1x^*,x^*)=0,$ i.e.  $T_1x^*=x^*$. In the same way it can be established that$T_2x^*=x^*$. Hence $T_1x^*=x^*=T_2x^*$. Thus $x^*$ is the common fixed point of pair of maps $T_1$ and $T_2$.
The uniqueness of the common fixed point is easy read from the contractive condition.

\end{proof}

\section{Admissibility Conditions}

In this section, we analyze the existence of fixed points for mapping
defined on a complete $K$-PMS $(X, p, K)$ using some admissibility conditions. The results we present are natural extensions of the Banach contraction principle.

We then introduce the following definitions.
\begin{definition}
	Let $(X,p,K)$ be a $K$-PMS, $f:X\to X$ and $\alpha,\beta: X \times X\to [0,\infty)$ be mappings and $C_\alpha>0, C_\beta\geq 0$. We say that $T$ is $(C_\alpha,C_\beta)$-admissible with respect to $K$ if the following conditions hold: 
	\begin{enumerate}
		\item[(C1)] $\alpha(x,y)\geq C_\alpha \Longrightarrow \alpha(fx,fy)\geq C_\alpha,$ whenever $x,y\in X$;
		\item[(C2)]$\beta(x,y)\leq C_\beta \Longrightarrow \beta(fx,fy)\leq C_\beta,$ whenever $x,y\in X$;
		\item[(C3)] $0\leq C_\beta/C_\alpha < 1/K$.
	\end{enumerate}
\end{definition}

\begin{definition}
	Let $(X,p,K)$ be a partial metric type space. A function $T:X\to X$ is called \textbf{$p$-sequentially continuous} if for any $\tau(p)$-convergent sequence $(x_n)$ with $x_n \overset{p}{\longrightarrow} x $, the sequence $(Tx_n)$ $\tau(p)$-converges to $Tx$, i.e. $Tx_n \overset{p}{\longrightarrow} Tx $.
\end{definition}

%
%
%

We now state our main fixed point theorem.

\begin{theorem}\label{fix1}
	Let $(X,p,K)$ be a Hausdorff complete partial metric type space. Suppose that $f:X\to X$ is $(C_\alpha,C_\beta)$-admissible with respect to $K$. Assume that
	\begin{equation}\label{cond}
	\alpha(x,y)p(fx,fy) \leq \beta(x,y)p(x,y) \qquad \text{for all } \ \ x,y \in X.
	\end{equation}
	
	If the following conditions hold:
	
	\begin{enumerate}
		\item[(i)] $f$ is $p$-sequentially continuous;
		\item[(ii)] there exists $x_0\in X$ such that $\alpha(x_0,fx_0)\geq C_\alpha$ and
		$\beta(x_0,fx_0)\leq C_\beta$.
	\end{enumerate}
	Then $f$ has a fixed point.
\end{theorem}

\begin{proof}
	Let $x_0\in X$ such that $\alpha(x_0,fx_0)\geq C_\alpha$ and $\beta(x_0,fx_0)\leq C_\beta$. Define the sequence $(x_n)$ by $x_n=f^nx_0=fx_{n-1}$. Without loss of generality, we can always assume that $x_n\neq x_{n+1}$ for all $n\in\mathbb{N}$, since if $x_{n_0}=x_{n_0+1}$ for some $n_0\in \mathbb{N},$ the proof is complete. Since $f$ is  $(C_\alpha,C_\beta)$-admissible with respect to $K$ and $\alpha(x_0,fx_0)= \alpha(x_0,x_1)\geq C_\alpha,$
	we deduce that $\alpha(x_1,x_2)= \alpha(fx_0,fx_1)\geq C_\alpha.$ By continuing this process, we get that
	$\alpha(x_n,x_{n+1})\geq C_\alpha,$ for all $n\geq 0$. Similarly, we establish that $\beta(x_n,x_{n+1})\leq C_\beta,$ for all $n\geq 0$. Using \eqref{cond}, we get 
	
	\begin{align*}
	C_\alpha p(x_n,x_{n+1}) & \leq  \alpha(x_{n-1},x_{n}) p(x_n,x_{n+1}) \\
	& \leq  \beta(x_{n-1},x_{n})  p(x_{n-1},x_{n}) \\
	& \leq C_\beta p(x_{n-1},x_n),
	\end{align*}
	and hence 
	\[
	p(x_n,x_{n+1}) \leq \frac{C_\beta}{C_\alpha} p(x_{n-1},x_n) \qquad \text{ for all } \ n\geq 1.
	\]

	 Since $0\leq  \frac{C_\beta}{C_\alpha} <1/K,$ we derive that $(x_n)$ is a Cauchy sequence (see Theorem \ref{theorem1}). 
	
	Since $(X,p,K)$ is complete and $T$ is $p$-sequentially continuous, there exists $x^*$ such that $x_n \overset{p}{\longrightarrow} x^*$ and $x_{n+1}\overset{p}{\longrightarrow} fx^*$. Since $X$ is Hausdorff, $x^*=fx^*$.
\end{proof}

We give the following results which are in fact consequences of the Theorem \ref{fix1}.

\begin{theorem}
	Let $(X,p,K)$ be a Hausdorff complete partial metric type space. Suppose that $f:X\to X$ is $(C_\alpha,C_\beta)$-admissible with respect to $K$. Assume that
	\begin{equation}\label{conds}
	\alpha(x,y)p(fx,fy) \leq \beta(x,y)p(x,y) \qquad \text{for all } \ \ x,y \in X.
	\end{equation}
	
	If the following conditions hold:
	
	\begin{enumerate}
		\item[(i)] there exists $x_0\in X$ such that $\alpha(x_0,fx_0)\geq C_\alpha$ and
		$\beta(x_0,fx_0)\leq C_\beta$;
		\item[(ii)] if $(x_n)$ is a sequence in $X$ such that $\alpha(x_n,x_{n+1})\geq C_\alpha$  and $\beta(x_n,x_{n+1})\leq C_\beta$ for all $n=1,2,\cdots$ and $x_n \overset{p}{\longrightarrow} x$, then there exists a subsequence $(x_{n(k)})$ of $(x_n)$ such that $\alpha(x,x_{n(k)})\geq C_\alpha$ and $\beta(x,x_{n(k)})\leq C_\beta$ for all $k$.
	\end{enumerate}
	Then $f$ has a fixed point.
\end{theorem}

\begin{proof}
	Following the proof of Theorem \ref{fix1}, we know that the sequence $(x_n)$ defined by 
	$x_{n+1}=fx_n$ for all $n=0,1,2,\cdots$ $p$-converges to some $x^*$ and satisfies $\alpha(x_n,x_{n+1})\geq C_\alpha$ and $\beta(x_n,x_{n+1})\leq C_\beta $ for $n\geq 1$. From the condition $(ii)$, we know there exists a subsequence $(x_{n(k)})$ of $(x_n)$ such that $\alpha(x^*,x_{n(k)})\geq C_\alpha$ and $\beta(x^*x_{n(k)})\leq C_\beta$ for all $k$. Since $f$ is satisfies \eqref{conds}, we get 
	
	\begin{align*}
	C_\alpha p(fx^*,x_{n(k)+1}) = C_\alpha p(fx^*,fx_{n(k)})& \leq  \alpha(x^*,x_{n(k)}) p(fx^*,fx_{n(k)}) \\
	& \leq   \beta(x^*,x_{n(k)})  p(x^*,x_{n(k)}) \\
	& \leq C_\beta  p(x^*,x_{n(k)})
	\end{align*}

	Letting $k\to \infty$, we obtain $p(fx^*,x_{n(k)+1}) \to 0$. Since $X$ is Hausdorff, we have that $fx^*=x^*$.
	This completes the proof.
\end{proof}

\section{The $\alpha$-series hypothesis}

In this section, we intend to extend the result of Choudhury\cite{bina1} by considering a family of self-maps on a partial metric type space. In particular, we make use of the so-called $\alpha$-series. We first recall the definition of an $\alpha$-series.


\begin{definition}\label{def1}(See \cite{Cal})
	
	Let $\{a_n \}$ be a sequence of non-negative real numbers. We say that the series $\sum_{n=1}^{\infty} a_n$ is an $\alpha$-series, if there exist $0<\lambda<1$ and $n(\lambda) \in \mathbb{N}$ such that
	
	$$\sum_{i=1}^{L} a_i \leq \lambda L \text{ for each } L\geq n(\lambda).$$
\end{definition}

\begin{remark}
	Each convergent series of non-negative real terms is an $\alpha$-series. However, there are also divergent series that are $\alpha$-series, and $\sum_{n=1}^{\infty}\frac{1}{n}$ is an instance.
\end{remark}

Also, we denote by $\Phi$ be the class of continuous, non-decreasing, sub-additive and homogeneous functions $F:[0,\infty) \to [0,\infty)$ such that $F^{-1}(0)=\{0\}$.

\subsection{Common fixed point theorems (Kannan-Choudhury case)}

\hspace*{2cm}	

\vspace*{0.2cm}

In this section, we prove existence of a unique common
fixed point for a family of contractive type self-maps on a
complete metric type space by using the Kannan contractive condition as a base.

\begin{theorem}\label{main}

	Let $(X,p,K)$ be a complete partial metric type space and $\{T_n\}$ be a sequence of self mappings on $X$ such that 
	
	\begin{align}\label{conditionqq}
	F(p(T_i(x),T_j(y))) \leq \  & F(\delta_{i,j}[p(x,T_i(x))+ p(y,T_j(y))])\\
	&- F(\gamma_{i,j}\psi[p(x,T_i(x)),p(y,T_j(y))]) \nonumber
	\end{align}
	
	for $x,y\in X$ with $x\neq y,$ $0\leq \delta_{i,j},\gamma_{i,j}<1 , \ i,j = 1,2,\cdots ,$ and for some $F \in \Phi$ homogeneous with degree $s$.
	and $\psi:[0,\infty)^2 \to [0,\infty)$ is a continuous mapping such that $%
	\psi(x,y)=0$ if and only if $x=y=0$. If the series $\sum_{i=1}^{\infty}s_i$ where $s_i=\frac{\delta_{i,i+1}^s }{1-\delta_{i,i+1}^s}$ is an 
	$\alpha$-series
	then $\{T_n\}$ has a unique common fixed point in $X$. 
	
\end{theorem}

\begin{proof}
	
	For any $x_0\in X$, we construct the Picard type sequence $(x_n)$ by setting $x_n= T_n(x_{n-1}),\ n=1,2,\cdots .$ Using \eqref{conditionqq} and the homogeneity of $F$, we obtain
	
	\begin{align*}
	F(p(x_1,x_2))  = \ & F(p(T_1(x_0),T_2(x_1))) \\
	\leq \ &\delta_{1,2}^s F([p(x_0,T_1(x_0)))+ p(x_1,T_2(x_1))]) \\
	& - \gamma_{1,2}^sF(\psi[p(x_0,T_1(x_0))), p(x_1,T_2(x_1))]) \\
	= \ & \delta_{1,2}^sF([p(x_0,x_1)+ p(x_1,x_2)] )-  \gamma_{1,2}^sF(\psi[p(x_0,x_1), p(x_1,x_2)])\\
	\leq & \delta_{1,2}^sF([p(x_0,x_1)+ p(x_1,x_2)] )
	\end{align*}
	
	Therefore, using the sub-additivity of $F$, we deduce that
	\[ (1-\delta_{1,2}^s)F(p(x_1,x_2)) \leq (\delta_{1,2}^s)   F(p(x_0,x_1)),\] 
	i.e.
	
	\[  F(p(x_1,x_2)) \leq \left(\frac{\delta_{1,2}^s}{1-\delta_{1,2}^s}   \right) F(p(x_0,x_1)) .\]
	
	Also, we get
	
	\begin{align*}
	F(p(x_2,x_3)) & = F(p(T_2(x_1),T_3(x_2))) \\
	& \leq \left(\frac{\delta_{2,3}^s}{1-\delta_{2,3}^s}   \right) F(p(x_1,x_2)) \\
	& \leq  \left(\frac{\delta_{2,3}^s}{1-\delta_{2,3}^s}   \right) \left(\frac{\delta_{1,2}^s}{1-\delta_{1,2}^s}   \right) F(p(x_0,x_1)).
	\end{align*}
	
	By repeating the above process, we have
	
	\begin{equation}
	F(p(x_n,x_{n+1})) \leq \prod\limits_{i=1}^n \left(\frac{\delta_{i,i+1}^s}{1-\delta_{i,i+1}^s}   \right)  F(p(x_0,x_1)).
	\end{equation}
	
	Hence we derive, by making use of the axiom (pm4) and the properties of $F$, that for $p>0$
	
	\begin{align*}
	F(p(x_n,x_{n+p})) \leq \ & K^s[F(p(x_n,x_{n+1})) +F(p(x_{n+1},x_{n+2})) \\
	& + \ldots + F(p(x_{n+p-1},x_{n+p}))]  \\
	\leq \ & K^s\left[\prod\limits_{i=1}^n \left(\frac{\delta_{i,i+1}^s}{1-\delta_{i,i+1}^s}   \right)  F(p(x_0,x_1)) \right. \\
	& + \prod\limits_{i=1}^{n+1} \left(\frac{\delta_{i,i+1}^s}{1-\delta_{i,i+1}^s}   \right)  F(p(x_0,x_1))        \\
	& + \ldots + \\
	& + \left.\prod\limits_{i=1}^{n+p-1} \left(\frac{\delta_{i,i+1}^s}{1-\delta_{i,i+1}^s}   \right)  F(p(x_0,x_1))\right] \\
	= \ & K^s\left[ \sum\limits_{k=0}^{p-1} \prod\limits_{i=1}^{n+k} \left(\frac{\delta_{i,i+1}^s}{1-\delta_{i,i+1}^s}   \right)  F(p(x_0,x_1))\right] \\
	= \ & K^s\left[\sum\limits_{k=n}^{n+p-1} \prod\limits_{i=1}^{k} \left(\frac{\delta_{i,i+1}^s}{1-\delta_{i,i+1}^s}   \right)  F(p(x_0,x_1))\right].
	\end{align*}

	Now, let $\lambda$ and $n(\lambda)$ as in Definition \ref{def1}, then for $n\geq n(\lambda)$ and using the fact that the geometric mean of non-negative real numbers is at most their arithmetic mean, it follows that 
	
	\begin{align}
	F(p(x_n,x_{n+p})) \leq \ & K^s\left[\sum\limits_{k=n}^{n+p-1} \left[ \frac{1}{k}\sum\limits_{i=1}^{k} \left(\frac{\delta_{i,i+1}^s}{1-\delta_{i,i+1}^s}   \right) \right]^k F(p(x_0,x_1))\right] \\
	\leq \ & K^s\left[ \left(\sum\limits_{k=n}^{n+p-1} \lambda^k \right) F(p(x_0,x_1))\right] \nonumber \\
	\leq \ & K^s \frac{\lambda^n}{1-\lambda}F(p(x_0,x_1)) \nonumber .
	\end{align}
	Letting $n\to \infty$ and since $F^{-1}(0)=\{0\}$ and $F$ is continuous, we deduce that $p(x_n,x_{n+p}) \to 0.$ Thus $(x_n)$ is a Cauchy sequence and, by completeness of $X$, converges to say $x^* \in X$, i.e. there exists  $x^* \in X$ such that
	
	$$ 0=\lim\limits_{n,m\to \infty} p(x_n , x_m)= \lim\limits_{n\to \infty} p(x_n , x^*)=p(x^*,x^*).$$
	
	Moreover, for any natural number $m\neq 0$, we have
	
	\begin{align*}
	F(p(x_n, T_m(x^*))) = \ & F(p(T_n(x_{n-1}),T_m(x^*))) \\
	\leq \ & \delta_{n,m}^s[F(p(x_{n-1},x_n))+ F(p(x^*,T_m(x^*)))] \\
	& - \gamma_{n,m}^sF(\psi[p(x_{n-1},x_n),p(x^*,T_m(x^*)]).
	\end{align*}
	
	Again, letting $n \to \infty$, we get
	
	\begin{align*}
	F(p(x^*, T_m(x^*))) & \leq \  \delta_{n,m}^s[F(p(x^*,x^*))+ F(p(x^*,T_m(x^*)))] - \gamma_{n,m}^sF(\psi[0,p(x^*,T_m(x^*)]) \\
	& \leq \ \delta_{n,m} ^sF(p(x^*,T_m(x^*)),
	\end{align*}
	and since $0\leq \delta_{n,m} <1$, it follows that $F(p(x^*, T_m(x^*)))=0$, i.e. $T_m(x^*)=x^*$.
	
	Then $x^*$ is a common fixed point of $\{T_m\}_{m\geq 1}$. 
	
	To prove the uniqueness of $x^*$, let us suppose that $y^*$ is a common fixed point of $\{T_m\}_{m\geq 1}$, that is $T_m(y^*)=y^*$ for any $m\geq 1$. Then, by \eqref{conditionqq}, we have
	
	\begin{align*}
	F(p(x^*,y^*))  & \leq \ F(p(T_m(x^*),T_m(y^*))) \\
	& \leq \ \delta_{n,m}^s [F(p(x^*,T_m(x^*)) + F(p(y^*,T_m(y^*))] - \gamma_{n,m}^sF(\psi[p(x^*,T_m(x^*), p(y^*,T_m(y^*))]) \\
	& = \delta_{n,m}^s [F(0) + F(0)] - \gamma_{n,m}^sF(\psi[0,0])\\
	& = 0.  
	\end{align*}
	
	So $x^*$ is the unique common fixed point of $\{T_m\}$. 
	
\end{proof}

As particular cases of Theorem \ref{main}, we have the following two corollaries.

\begin{corollary}\label{cor1}
	Let $(X,p,K)$ be a complete metric type space and $\{T_n\}$ be a sequence of self mappings on $X$ such that
	
	\begin{align}\label{condition}
	p(T_i(x),T_j(y)))\leq \  & \delta_{i,j}[p(x,T_i(x))+ p(y,T_j(y))]\\
	&- \gamma_{i,j}\psi[p(x,T_i(x)),p(y,T_j(y))] \nonumber
	\end{align}
	
	for $x,y\in X$ with $x\neq y,$ $0\leq \delta_{i,j},\gamma_{i,j}<1 , \ i,j = 1,2,\cdots ,$ 
	and $\psi:[0,\infty)^2 \to [0,\infty)$ is a continuous mapping such that $%
	\psi(x,y)=0$ if and only if $x=y=0$. If the series $\sum_{i=1}^{n}(s_i)$ where $s_i=\frac{\delta_{i,i+1}^s}{1-\delta_{i,i+1}^s}$ is an $\alpha$-series, 
	then $\{T_n\}$ has a unique common fixed point in $X$. 
	
\end{corollary}

\begin{proof} 
	Apply Theorem \ref{main} by putting $F=I_{[0,\infty)}$, the identity map.	
\end{proof}

\begin{corollary}\label{cor2}
	Let $(X,p,K)$ be a complete metric type space and $\{T_n\}$ be a sequence of self mappings on $X$ such that 
	
	\begin{align}\label{condition2wwwwww}
	F(p(T_i(x),T_j(y))) \leq \  & F(\delta_{i,j}[p(x,T_i(x))+ p(y,T_j(y))])
	\end{align}
	
	for $x,y\in X$ with $x\neq y,$ $0\leq \delta_{i,j}<1 , \ i,j = 1,2,\cdots ,$ and for some $F \in \Phi$ homogeneous with degree $s$.
	If the series $\sum_{i=1}^{n}(s_i)$ where $s_i=\frac{\delta_{i,i+1}^s}{1-\delta_{i,i+1}^s}$ is an $\alpha$-series, 
	then $\{T_n\}$ has a unique common fixed point in $X$. 
	
\end{corollary}

\begin{proof} 
	Apply Theorem \ref{main} by putting\footnote{In this case, we can choose $(b_n)$ to be any constant sequence of elements of $X$.} $\gamma_{i,j}=0$.
\end{proof}

A more general\footnote{Natural in some sense} weak contraction could also be considered as we write our next result.

\begin{theorem}\label{main2}
	Let $(X,p,K)$ be a complete metric type space and $\{T_n\}$ be a sequence of self mappings on $X$ such that
	
	\begin{align}\label{condition2}
	F(p(T_i(x),T_j(y))) \leq \  & F(\delta_{i,j}[p(x,T_i(x))+ p(y,T_j(y))+p(x,y)])\\
	&- F(\gamma_{i,j}\psi[p(x,T_i(x)),p(y,T_j(y)),p(x,y)]) \nonumber
	\end{align}
	
	for $x,y\in X$ with $x\neq y,$ $0\leq \delta_{i,j},\gamma_{i,j}<1 , \ i,j = 1,2,\cdots ,$ and for some $F \in \Phi$ homogeneous with degree $s$,
	and $\psi:[0,\infty)^3 \to [0,\infty)$ is a continuous mapping such that $%
	\psi(x,y,z)=0$ if and only if $x=y=z=0$.If the series $\sum_{i=1}^{n}(s_i)$ where $s_i=\frac{\delta_{i,i+1}^s}{1-\delta_{i,i+1}^s}$ is an $\alpha$-series, 
	then $\{T_n\}$ has a unique common fixed point in $X$.  
	
\end{theorem}

We shall omit the proof as it is merely a copy of that of Theorem \ref{main}. Then we also give, without proof the following corollary.

\begin{corollary}\label{cor11}
	Let $(X,D,K)$ be a complete metric type space and $\{T_n\}$ be a sequence of self mappings on $X$ such that
	
	\begin{align}\label{condition22}
	F(p(T_i(x),T_j(y)))\leq \  &F( \delta_{i,j}[p(x,T_i(x))+ p(y,T_j(y))+D(x,y)]),\\
	\end{align}
	
	for $x,y\in X$ with $x\neq y,$ $0\leq \delta_{i,j}<1 , \ i,j = 1,2,\cdots ,$ where 
	 $\psi:[0,\infty)^3 \to [0,\infty)$ is a continuous mapping such that $%
	\psi(x,y,z)=0$ if and only if $x=y=z=0$. If the series $\sum_{i=1}^{n}(s_i)$ where $s_i=\frac{\delta_{i,i+1}^s}{1-\delta_{i,i+1}^s}$ is an $\alpha$-series, 
	then $\{T_n\}$ has a unique common fixed point in $X$. 
	
\end{corollary}

		\begin{example}
		Let $X=[0,1]$ and $p$ be the metric type defined by $p(x,y)= (\max\{x,y\})^2$ for $x,y \in X$. Clearly, $(X,p,2)$ is a complete partial metric type space. 
		Following the notation in Theorem \ref{main}, we set 
		 $\delta_{i,j}=\left(\frac{1}{1+2^\eta}\right)^2$ where $\eta= \min\{i,j\}$. We also define $T_i(x)= \frac{x}{16^i}$ for all $x\in X$ and $i=1,2,\cdots$ and $F:[0,\infty) \to [0,\infty), \ x\mapsto \sqrt{x}$. Then $F$ is continuous, non-decreasing, sub-additive and homogeneous with degree $s=\frac{1}{2}$ and $F^{-1}(0)=\{0\}$.

		Assume $i<j$ and $x>y$. Hence we have
		
		\[ F(p(T_i(x),T_j(y)))  = \frac{x}{16^i} \]
		and
		
		\[ F( \delta_{i,j} [p(x,T_i(x))+ p(y,T_j(y))+p(x,y)]  ) =  \sqrt { \left(\frac{1}{1+2^i}\right)^2 (2x^2+y^2)} .\]
		
		Therefore condition \eqref{condition22} is satisfied for all $x,y\in X$ with $x\neq y$. Moreover, since $F$ is homogeneous with degree $s=\frac{1}{2}$, the sequence
		$$ s_i=\frac{ 2^s \delta_{i,i+1}^s }{1-\delta_{i,i+1}^s} = \frac{\sqrt{2}}{2^i} $$
		is a $\lambda$-sequence with $\lambda=\frac{\sqrt{2}}{2}$. Then by Corollary \ref{cor11}, $\{T_n\}$ has a common fixed point, which is this case $x^*=0.$
	\end{example}

The contractive condition in Theorem \ref{main} can be relaxed and we obtain the next result: 

\begin{theorem}\label{relaxed}
	Let $(X,D,K)$ be a complete partial metric type space and $\{T_n\}$ be a sequence of self mappings on $X$ such that
	
	\begin{align}\label{conditionrelaxed}
	F(p(T_i(x),T_j(y))) \leq \  & F(\delta_{i,j}[p(x,T_i(x))+ p(y,T_j(y))])\\
	&- F(\gamma_{i,j}\psi[p(x,T_i(x)),p(y,T_j(y))]) \nonumber
	\end{align}
	
	for $x,y\in X$ with $x\neq y,$ $0\leq \delta_{i,j}, \gamma_{i,j} , \ i,j = 1,2,\cdots ,$ for some $F \in \Phi$ homogeneous with degree $s$ and $\psi:[0,\infty)^2 \to [0,\infty)$ is a continuous mapping such that $%
	\psi(x,y)=0$ if and only if $x=y=0$.
	If 
	\begin{itemize}
		\item[i)] for each $j,$ $\underset{i \to \infty}{\limsup}\ \delta_{i,j}^s <1, $
		
		\item[ii)] $$\sum_{n=1}^{\infty} C_n < \infty \text{ where } C_n = \prod\limits_{i=1}^n \frac{\delta_{i,i+1}^s}{1-\delta_{i,i+1}^s}, $$
	\end{itemize}

	then $\{T_n\}$ has a unique common fixed point in $X$. 
	
\end{theorem}

\begin{proof}
	Just observe that the Picard type sequence $(x_n)=(T_n(x_{n-1}))$ for any initial $x_0\in X$ is such that 
	
	\begin{equation}
	F(p(x_n,x_{n+1})) \leq \prod\limits_{i=1}^n \left(\frac{\delta_{i,i+1}^s}{1-\delta_{i,i+1}^s}   \right)  F(p(x_0,x_1)) =: C_n F(p(x_0,x_1)),
	\end{equation}
	and so for $p>1$
	
	\begin{align*}
	F(p(x_n,x_{n+p})) \leq K^{s}\left[\sum\limits_{k=n}^{n+p-1} C_k\right] F(p(x_0,x_1)).
	\end{align*}
	Letting $n\to \infty$, we deduce that $p(x_n,x_{n+p}) \to 0.$ Thus $(x_n)$ is a Cauchy sequence and, by completeness of $X$, converges to say $x^* \in X$. It is easy to see that $x^*$ is the unique common fixed point of $\{T_m\}$.	
	
\end{proof}

\begin{corollary}\label{correlax}
	Let $(X,D,K)$ be a complete partial metric type space and $\{T_n\}$ be a sequence of self mappings on $X$ such that
	
	\begin{align}\label{condcorrelaxed}
	F(p(T_i(x),T_j(y))) \leq \  & F(\delta_{i,j}[p(x,T_i(x))+ p(y,T_j(y))])
	\end{align}
	
	for $x,y\in X$ with $x\neq y,$ $0\leq \delta_{i,j}<1 , \ i,j = 1,2,\cdots ,$ for some $F \in \Phi$ homogeneous with degree $s$ and $\psi:[0,\infty)^2 \to [0,\infty)$ is a continuous mapping such that $%
	\psi(x,y)=0$ if and only if $x=y=0$.
	
	If 
	\begin{itemize}
		\item[i)] for each $j,$ $\underset{i \to \infty}{\limsup}\ \delta_{i,j}^s <1, $
		
		\item[ii)] $$\sum_{n=1}^{\infty} C_n < \infty \text{ where } C_n = \prod\limits_{i=1}^n \frac{\delta_{i,i+1}^s}{1-\delta_{i,i+1}^s}, $$
	\end{itemize}
	
	then $\{T_n\}$ has a unique common fixed point in $X$. 	
\end{corollary}		

\begin{proof}
	Apply Theorem \ref{relaxed} by putting $\gamma_{i,j}=0$.
\end{proof}

\begin{example}

	Let $X=[0,1]$ and $p(x,y)= |x-y|^2$ whenever $x,y \in [0,1]$. Clearly, $(X,p,2)$ is a complete partial metric type space. 
	
	Following the notation in Theorem \ref{relaxed}, we set $\delta_{i,j} = \left(\frac{1}{1+2^i}\right)^2$.
	We also define $T_i(x)= \frac{x}{4^i}$ for all $x\in X$ and $i=1,2,\cdots$ and $F:[0,\infty) \to [0,\infty), \ x\mapsto \sqrt{x}$. Then $F$ is continuous, non-decreasing, sub-additive and homogeneous of degree $s=\frac{1}{2}$ and $F^{-1}(0)=\{0\}$. Assume $i<j$ and $x>y\geq z$. Hence we have
	\[
	F(p(T_i(x),T_j(y))  = \sqrt{\left|\frac{x}{4^i}- \frac{y}{4^j}     \right|^2}
	\]
	
	and
	
	\[ F( \delta_{i,j} [p(x,T_i(x))+ p(y,T_j(y))]  ) =  \sqrt { \left(\frac{1}{1+2^i}\right)^2 \left(\left|x-\frac{x}{4^i}\right|^2+ \left| y-\frac{y}{4^j} \right|^2 \right)} .\]
	
	Therefore condition \eqref{condcorrelaxed} is satisfied for all $x,y\in X$ with $x\neq y$. Moreover, since $F$ is homogeneous of degree $s=\frac{1}{2}$, we have 
	\begin{itemize}
		\item[i)] $\underset{i \to \infty}{\limsup}\ \delta^s_{i,j} <1, $

		\item[ii)] $$ C_n= \prod\limits_{i=1}^n \frac{\delta_{i,i+1}^s}{1-\delta_{i,i+1}^s} =  \prod\limits_{i=1}^n \frac{1}{2^i}=  \frac{1}{2^{\frac{n(n+1)}{2}}} \leq \frac{1}{2 ^n}.$$
		
	\end{itemize}
	
	The conditions of Corollary \ref{correlax} are satisfied, $\{T_n\}$ has a common fixed point, which is this case $x^*=0.$

\end{example}

\subsection{Common fixed point theorems (Chatterjea-Choudhury case)} \hspace*{2cm}	

\vspace*{0.2cm}

In this section, we prove existence of a unique common
fixed point for a family of contractive type self-maps on a
complete partial metric type space by using the Chatterjea contractive condition as a base.

\begin{theorem}\label{chat1}
	
	Let $(X,p,K)$ be a complete partial metric type space and $\{T_n\}$ be a sequence of self mappings on $X$ such that
	
	\begin{align}\label{condchat1}
	F(p(T_i(x),T_j(y))) \leq \  & F(\delta_{i,j}[p(x,T_j(y))+ p(y,T_i(x))])\\
	&- F(\gamma_{i,j}\psi[p(x,T_j(y)),p(y,T_i(x))]) \nonumber
	\end{align}
	
	for $x,y\in X$ with $x\neq y,$ $0\leq \delta_{i,j},\gamma_{i,j}<1 , \ i,j = 1,2,\cdots ,$ and for some $F \in \Phi$ homogeneous with degree $s$.
	and $\psi:[0,\infty)^2 \to [0,\infty)$ is a continuous mapping such that $%
	\psi(x,y)=0$ if and only if $x=y=0$. If the series $\sum_{i=1}^{\infty}s_i$ where $s_i=\frac{\delta_{i,i+1}^s }{1-\delta_{i,i+1}^s}$ is an 
	$\alpha$-series
	then $\{T_n\}$ has a unique common fixed point in $X$.

\end{theorem}

\begin{proof}
	The proof follows exactly the same steps as the proof of Theorem \ref{main}
\end{proof}

\begin{theorem}\label{chat2}
	Let $(X,p,K)$ be a complete partial metric type space and $\{T_n\}$ be a sequence of self mappings on $X$ such that
	
	\begin{align}\label{condchat2}
	F(p(T_i(x),T_j(y))) \leq \  & F(\delta_{i,j}[p(x,T_j(x))+ p(y,T_i(y))+p(x,y)])\\
	&- F(\gamma_{i,j}\psi[p(x,T_j(y)),p(y,T_i(x)),p(x,y)]) \nonumber
	\end{align}
	
	for $x,y\in X$ with $x\neq y,$ $0\leq \delta_{i,j},\gamma_{i,j}<1 , \ i,j = 1,2,\cdots ,$ and for some $F \in \Phi$ homogeneous with degree $s$, where 
	and $\psi:[0,\infty)^3 \to [0,\infty)$ is a continuous mapping such that $%
	\psi(x,y,z)=0$ if and only if $x=y=z=0$. If the series $\sum_{i=1}^{\infty}s_i$ where $s_i=\frac{\delta_{i,i+1}^s }{1-\delta_{i,i+1}^s}$ is an 
	$\alpha$-series
	then $\{T_n\}$ has a unique common fixed point in $X$. 
\end{theorem}

\begin{proof}
	The proof follows exactly the same steps as the proof of Theorem \ref{main2}
\end{proof}

Like in the case of Theorem \ref{main}, the condition \eqref{chat1} can also be relaxed and we obtain:

\begin{theorem}\label{chat1relaxed}
	Let $(X,p,K)$ be a complete partial metric type space and $\{T_n\}$ be a sequence of self mappings on $X$ such that
	
	\begin{align}\label{condchat1relaxed}
	F(p(T_i(x),T_j(y))) \leq \  & F(\delta_{i,j}[p(x,T_j(y))+ p(y,T_i(x))])\\
	&- F(\gamma_{i,j}\psi[p(x,T_j(y)),p(y,T_i(x))]) \nonumber
	\end{align}
	
	for $x,y\in X$ with $x\neq y,$ $0\leq \delta_{i,j}, \gamma_{i,j}, \ i,j = 1,2,\cdots ,$ for some $F \in \Phi$ homogeneous with degree $s$ and $\psi:[0,\infty)^2 \to [0,\infty)$ is a continuous mapping such that $%
	\psi(x,y)=0$ if and only if $x=y=0$.
	If 
	\begin{itemize}
		\item[i)] for each $j,$ $\underset{i \to \infty}{\limsup}\ \delta_{i,j}^s <1, $
		
		\item[ii)] $$\sum_{n=1}^{\infty} C_n < \infty \text{ where } C_n = \prod\limits_{i=1}^n \frac{\delta_{i,i+1}^s}{1-\delta_{i,i+1}^s}, $$
	\end{itemize}
	
	then $\{T_n\}$ has a unique common fixed point in $X$. 
	
\end{theorem}

Also, it is not necessary to point out that this relaxed condition can be applied to modify the hypotheses of Theorem \ref{chat2}.

\begin{corollary}
	In addition to hypotheses of Theorem \ref{chat1relaxed},
	suppose that for every $n \geq 1$, there exists $k(n) \geq 1 $ such
	that $\delta_{n,k(n)}< \frac{1}{2}$; then every $T_n$ has a unique fixed point in $X$.
\end{corollary}

\begin{remark}
	This corollary emphasizes on the subtle fact that in addition to the existence of a unique common fixed for the family $\{T_n\}$, each single map $T_n$ has a unique fixed point, which happens to be the common one. 
\end{remark}

\begin{proof}
	From Theorem \ref{chat1relaxed}, we know that the family $\{T_n\}$ has a
	unique common fixed point $x^* \in X$. If $y^*$ is a fixed
	point of a given $T_m$ then
	\begin{align*}
	p(x^*,y^*) = p(T_{k(m)}x^*,T_my^*) & \leq \delta_{k(m),m}[p(x^*,T_my^*)+ p(y^*,T_{k(m)}x^*)] \\
	& = \delta_{k(m),m}[p(x^*,y^*)+ p(y^*,x^*)] \\
	& < p(x^*,y^*),
	\end{align*}
	
	which implies $p(x^*,y^*) =0 $; which gives
	the desired result.
	
\end{proof}

\begin{example}
	We endow the set $X=[0,1]$ with the partial metric type $p(x,y)= |x-y|$ for $x,y,\in X$. Then $(X,p,1)$ is a complete partial metric type space. We set $F(x)=x$ whenever $x\in X$, then $F \in \Phi$ homogeneous with degree 1. For $n\geq$ we define the family $\{T_n\}$ of self-mappings on $X$ by
	
	\[T_n= \left\{
	\begin{array}{ll}
	1 & : 0 < x \leq 1,\\
	\frac{2}{3} + \frac{1}{n+2} & : x = 0.
	\end{array}
	\right.
	\]
	
	Using the notation of Theorem \ref{chat1relaxed}, we use the family of reals 
	$$\delta_{i,j} = \frac{1}{3} + \frac{1}{|i-j|+6};\qquad \gamma_{i,j}=0.$$
	
	Then clearly for each $j,$ 
	$$\underset{i \to \infty}{\limsup}\ \delta_{i,j}^s <1, $$ 
	
	and

	$$C_n = \prod \limits_{i=1}^n \frac{\delta_{i,i+1}^s}{1-\delta_{i,i+1}^s} =  \left(\frac{10}{11}\right)^n ,$$
	which is a convergent series.
	
	Also it is a straightforward calculation to verify that 
	
	\begin{align}
	F(p(T_i(x),T_j(y))) \leq \  & F(\delta_{i,j}[p(x,T_j(y))+ p(y,T_i(x))])\\
	&- F(\gamma_{i,j}\psi[p(x,T_j(y)),p(y,T_i(x))]) \nonumber
	\end{align}
	
	i.e. equivalently 
	
	\begin{align}
	p(T_i(x),T_j(y))\leq \  & \delta_{i,j}[p(x,T_j(y))+ p(y,T_i(x))].
	\end{align}
	
	One just has to consider the three possible cases: 
	
	\begin{itemize}
		\item $x\in (0,1]$ and $y\in (0,1]$;
		\item $x\in (0,1]$ and $y=0$;
		\item $x=y=0$ with $i\neq j$.
	\end{itemize}
	
	So all the conditions of Theorem \ref{chat1relaxed} are satisfied and note that $x = 1$ is the only fixed point for the $T_n$'s.
\end{example}

\section{Conclusion and future work}
From the above results, it is very clear that $K$-PMS represent quite a large class of ``metric like" structures and such setting can be made even larger if we look at the work by Choi \cite{choi} and Gaba\cite{gabs}.

Using the same idea as in the proofs of Theorem \ref{main} and Theorem \ref{main2}, one can establish the
following results.

\begin{theorem}
	
	Let $(X,D,K)$ be a complete metric type space and $\{T_n\}$ be a sequence of self mappings on $X$. Assume that there exist two sequences $(a_n)$ and $(b_n)$ of elements of $X$ such 
	
	\begin{align}
	F(p(T_i^r(x),T_j^r(y))) \leq \  & F(\delta_{i,j}[p(x,T_i^r(x))+ p(y,T_j^r(y))])\\
	&- F(\gamma_{i,j}\psi[p(x,T_i^r(x)),p(y,T_j^r(y))]) \nonumber
	\end{align}
	
	for $x,y\in X$ with $x\neq y,$ $r\geq 1,$ $0\leq \delta_{i,j},\gamma_{i,j}<1 , \ i,j = 1,2,\cdots ,$ and for some $F \in \Phi$ homogeneous with degree $s$,
	 where 
	 $\psi:[0,\infty)^2 \to [0,\infty)$ is a continuous mapping such that 
	$\psi(x,y)=0$ if and only if $x=y=0$. 
	then $\{T_n\}$ has a unique common fixed point in $X$.

	If the series $\sum_{i=1}^{\infty}s_i$ where $s_i=\frac{\delta_{i,i+1}^s }{1-\delta_{i,i+1}^s}$ is an 
	$\alpha$-series
	then $\{T_n\}$ has a unique common fixed point in $X$.
\end{theorem}	

\begin{theorem}
	Let $(X,p,K)$ be a complete metric type space and $\{T_n\}$ be a sequence of self mappings on $X$ such that 
	\begin{align}
	F(p(T_i^r(x),T_j^r(y))) \leq \  & F(\delta_{i,j}[p(x,T_i^r(x))+ p(y,T_j^r(y))+p(x,y)])\\
	&- F(\gamma_{i,j}\psi[p(x,T_i^r(x)),p(y,T_j^r(y)),p(x,y)]) \nonumber
	\end{align}
	for $x,y\in X$ with $x\neq y,$ $r\geq 1,$ $0\leq \delta_{i,j},\gamma_{i,j}<1 , \ i,j = 1,2,\cdots ,$ and for some $F \in \Phi$ homogeneous with degree $s$, 
	 $\psi:[0,\infty)^3 \to [0,\infty)$ is a continuous mapping such that $\psi(x,y,z)=0$ if and only if $x=y=z=0$. 
	
	If the series $\sum_{i=1}^{\infty}s_i$ where $s_i=\frac{2^s\delta_{i,i+1}^s }{1-\delta_{i,i+1}^s}$ is an 
$\alpha$-series then $\{T_n\}$ has a unique common fixed point in $X$. 
	\end{theorem}

Moreover, the above two generalisations also apply to Theorem \ref{chat1} and Theorem \ref{chat2} respectively.	

\vspace*{0.5cm}	

Recently the so-called $C$-class functions, which was introduced by A. H. Ansari \cite{aha} in 2014 and covers a large class of contractive conditions, has been applied successfully in the generalization of certain contractive conditions.

We read
\begin{definition}
	\label{C-class}(\cite{aha}) A continuous function $f:[0,\infty
	)^{2}\rightarrow \mathbb{R}$ is called \textit{$C$-class} function if \ for
	any $s,t\in \lbrack 0,\infty ),$ the following conditions hold:
	
	(1) $f(s,t)\leq s$;
	
	(2) $f(s,t)=s$ implies that either $s=0$ or $t=0$.
	
	We shall denote by $\mathcal{C}$ the collection of $C$-class functions.	
\end{definition}

\begin{example}
	\label{C-class examp}(\cite{aha}) The following examples show that the class $%
	\mathcal{C}$ is nonempty:

	\begin{enumerate}
		\item $f(s,t)=s-t.$
		
		\item $f(s,t)=ms,$ for some $m\in (0,1).$
		
		\item $f(s,t)=\frac{s}{(1+t)^{r}}$ for some $r\in (0,\infty ).$
		
		\item $f(s,t)=\log (t+a^{s})/(1+t)$, for some $a>1.$
		
	\end{enumerate}

\end{example}	

Therefore, the authors plan to study, in another manuscript, the existence of common fixed points for a family of mappings $T_i: (X,p,K)\to (X,p,K), i=1,2,\cdots$, defined on a partial metric type space $(X,p,K)$, which satisfy :

\begin{align}\label{aha}
F(p(T_i(x),T_j(y))) \leq \  & f(F(\delta_{i,j}[p(x,T_i(x))+ p(y,T_j(y))+p(x,y)]),\\
& F(\gamma_{i,j}\psi[p(x,T_i(x)),p(y,T_j(y)),p(x,y)])) \nonumber
\end{align}
for $x,y\in X$ with $x\neq y,$ where 

\begin{enumerate}
	
	\item $f \in \mathcal{C}$,
	\item  
	$0\leq \delta_{i,j},\gamma_{i,j}<1 , \ i,j = 1,2,\cdots ,$
	\item $F \in \Phi$ homogeneous with degree $s$
	
	\item $\psi:[0,\infty)^3 \to [0,\infty)$ is a continuous mapping such that $%
	\psi(x,y,z)=0$ if and only if $x=y=z=0$.
\end{enumerate}

under the condition that the series $\sum_{i=1}^{\infty}s_i$ where $s_i=\frac{2^s\delta_{i,i+1}^s }{1-\delta_{i,i+1}^s}$ is an 
$\alpha$-series.



{\bf Conflict of interest.}

The author declares that there is no conflict
of interests regarding the publication of this article.

{\bf Acknowledgments}

This work was carried out with financial support from the government of Canada’s International Development Research Centre (IDRC) and within the framework of the AIMS Research for Africa Project.

\bibliographystyle{amsplain}

\end{document}